\documentclass[12pt,twoside,reqno]{amsart}
\usepackage[all]{xy}
        \usepackage {amssymb,latexsym,amsthm,amsmath,mathtools}
        \usepackage{enumitem,color}
        
\usepackage{mathtools}

\usepackage[]{biblatex}
\long\def\symbolfootnote[#1]#2{\begingroup%
\def\thefootnote{\fnsymbol{footnote}}\footnote[#1]{#2}\endgroup}

\newcommand{\diag}{\textup{diag}}

\makeatletter
\def\imod#1{\allowbreak\mkern10mu({\operator@font mod}\,\,#1)}
\makeatother

\newtheorem{theorem}{Theorem}[section]
\newtheorem{lemma}[theorem]{Lemma}
\newtheorem{corollary}[theorem]{Corollary}
\newtheorem{proposition}[theorem]{Proposition}
\newtheorem*{theorem*}{Theorem}
\theoremstyle{definition}
\newtheorem{definition}[theorem]{Definition}
\newtheorem{remark}[theorem]{Remark}
\newtheorem{example}[theorem]{Example}
\numberwithin{equation}{section}

\newcommand{\ignore}[1]{}

\newcommand{\mynote}[1]{}
\begin{document}
\setcounter{section}{0}
\title{Surjectivity of polynomial maps on Matrices}
\author[Panja S.]{Saikat Panja}
\address{Harish-Chandra Research Institute Prayagraj (Allahabad), Uttar Pradesh 211019, India}
\email{panjasaikat300@gmail.com}
\author[Saini P.]{Prachi Saini}
\address{IISER Pune, Dr. Homi Bhabha Road, Pashan, Pune 411 008, India}
\email{prachi2608saini@gmail.com}
\author[Singh A.]{Anupam Singh}
\address{IISER Pune, Dr. Homi Bhabha Road, Pashan, Pune 411 008, India}
\email{anupamk18@gmail.com}
\thanks{The first-named author is supported by the HRI PDF-M scholarship. The second-named author acknowledges the support of CSIR PhD scholarship number 09/0936(1237)/2021-EMR-I. The third-named author is funded by an NBHM research grant 02011/23/2023/NBHM(RP)/RDII/5955 for this research.}
\today


\begin{abstract}
For $n\geq 2$, we consider the map on $M_n(\mathbb K)$ given by evaluation of a polynomial $f(X_1, \ldots, X_m)$ over the field $\mathbb K$. In this article, we explore the image of the diagonal map given by $f=\delta_1 X_1^{k_1} + \delta_2 X_2^{k_2} + \cdots +\delta_m X_m^{k_m}$ in terms of the solution of certain equations over $\mathbb K$. In particular, we show that for $m\geq 2$, the diagonal map is surjective when (a) $\mathbb K= \mathbb C$, (b) $\mathbb K= \mathbb F_q$ for large enough $q$. Moreover, when $\mathbb K= \mathbb R$ and $m=2$ it is surjective except when $n$ is odd, $k_1, k_2$ are both even, and $\delta _1\delta_2>0$ (in that case the image misses negative scalars), and the map is surjective for $m\geq 3$. We further show that on $M_n(\mathbb H)$ the diagonal map is surjective for $m\geq 2$, where $\mathbb H$ is the algebra of Hamiltonian quaternions.    
\end{abstract}

\maketitle

\section{Introduction}
Let $\mathcal A$ be an associative algebra over a field $\mathbb K$. Let $f(X_1, \ldots, X_m)$ be a polynomial over $\mathbb K$ in non-commuting variables. Such a polynomial defines a map $\omega \colon \mathcal A^m \rightarrow \mathcal A$ given by evaluation $(x_1, x_2, \ldots,x_m) \mapsto f(x_1, \ldots, x_m)$. These maps are called polynomial maps, and a fundamental question is understanding their image. In recent years this problem has attracted a lot of attention including in the context of groups, algebras and Lie algebras etc. (e.g. see the articles~\cite{BelovKunyavskiiPlotkin2013, BelovMalevRowenYavich2020} ). Several deep results have been proved for word maps on finite simple groups (see the articles~\cite{LarsenShalev2009, LarsenShalevTiep2011, LarsenShalevTipe2019}). In this article, we look for certain analogous results for matrix algebras. Here we deal with a particular map, namely given by diagonal polynomials. 

Given integers $k_1, k_2, \ldots, k_m \geq 1$, and $\delta_1, \ldots, \delta_m \in \mathbb K$ all non-zero, we consider the diagonal polynomial
\begin{equation}\label{multi-form} 
f(X_1,\ldots, X_m)=\delta_1 X_1^{k_1} + \delta_2 X_2^{k_2} + \cdots +\delta_m X_m^{k_m}
\end{equation}
in $m$-variables. We call the corresponding map $\omega$ a {\bf diagonal map}. The main question here is to understand how big the image is. In particular, we would like to see if this map is surjective and find the smallest $m$ with this property. When $k_1=\ldots = k_m = k$ the Equation~\ref{multi-form} is a $k$-form. Such a $k$-form is said to be {\bf universal} on $\mathcal A$ if the map $\omega$ is surjective. Also, the well-known matrix Waring problem considers $\delta_1= \delta_2 = \cdots = \delta_m = 1$ for a $k$-form and asks for the smallest $m$ so that the form is universal, i.,e., $\omega$ is surjective. The universality of the quadratic form ($k=2$ case) over a field is a well-studied problem including the arithmetic aspects (e.g. in the work of Bhargava~\cite{Bhargava2000}, \cite{BhargavaCremonaFisher2016}). Voloch looked at diagonal forms over function fields in~\cite{Voloch1985}. We are interested in looking at this problem when $\mathcal A$ is the matrix algebra $M_n(\mathbb K)$ and obtain results over certain fields such as $\mathbb K=\mathbb C, \mathbb R$ and $\mathbb F_q$ (also over the skew-field $\mathbb H$). The Waring problem is well-studied for matrix algebra see Richman \cite{Richman1987}, Vaserstein \cite{Vaserstein1987}, Katre \cite{KatreKrishnamurthi2022}, \cite{KatreWadikar2021}, Garge  \cite{BaraiGarge2022}, \cite{Garge2021}, 
Bresar-Semrl (\cite{Bresar2020}, \cite{BresarSemerl2022}, \cite{BresarSemerl2021}) etc. More general problems about the images of polynomial maps on algebras are being considered by Kanel-Below, Yavich, Kunyavskii, Rowen etc. (see \cite{BandmanGordeevKunyavaskii2012}, \cite{BelovMalevRowen2017}
,\cite{BelovMalevRowen2016}, \cite{BelovMalevRowen2012}, \cite{Lee2018}, \cite{PanjaPrasad2023}). For a survey of recent results one can look into \cite{BelovKunyavskiiPlotkin2013} and  \cite{BelovMalevRowenYavich2020}.

Motivated by the results obtained for the Waring-like problem in the case of non-abelian finite simple groups (see \cite{LarsenShalev2009}, \cite{LarsenShalevTiep2011} and \cite{LarsenShalevTipe2019}), Larsen asked the question (see \cite{Kishore2022}) if every element of $M_n(\mathbb K)$ can be written as a sum of two $k$-th powers for ``large enough'' $\mathbb K$. This question is answered in the affirmative in a series of two papers~\cite{Kishore2022} and~\cite{KishoreSingh2022} over finite fields $\mathbb F_q$. In particular, in~\cite{KishoreSingh2022} it is shown that (see Theorem 1.1) there exists a constant $\mathcal N(k)$ (which depends only on $k$) such that for all $q> \mathcal N(k)$ the map given by $X^k + Y^k$ is surjective on $M_n(\mathbb F_q)$. 

In this article, we obtain more general results about the image of the diagonal map $\omega$ given by the polynomial in Equation~\ref{multi-form} on $M_n(\mathbb K)$ such as when invertible matrices and nilpotent matrices are in the image depending on if we have solutions to certain equations (which are usually $k$-forms) over the underlying field $\mathbb K$. These results are collected in Section~\ref{diagonal-word}. With the help of these, we obtain the following results in Section~\ref{diagonal-fieldC}, and \ref{diagonal-fieldreal} over various fields: 
\begin{theorem}\label{theoremB}
Let $m$ be a positive integer and $n\geq 2$. Given integers $k_1, k_2, \ldots, k_m \geq 1$, and $\delta_1, \ldots, \delta_m \in \mathbb K$ all non-zero, consider the diagonal map $\omega \colon M_n(\mathbb K)^m \rightarrow M_n(\mathbb K)$ for $n\geq 2$ given by 
$$
\omega(x_1, \ldots, x_m) = \delta_1 x_1^{k_1} + \delta_2 x_2^{k_2} + \cdots +\delta_m x_m^{k_m}.$$
Then we have the following:
\begin{enumerate}
\item When $\mathbb K=\mathbb C$, the map $\omega$ is surjective for all $m\geq 2$. 
\item When $\mathbb K=\mathbb F_q$ and $m\geq 2$, there exists a constant $\mathcal N$ which depends on $k_1, \ldots, k_m$ such that for all $q >  \mathcal N$ the map $\omega$ is surjective. 
\item When $\mathbb K=\mathbb R$ and $m=2$, the map $\omega$ is surjective except when $n$ is odd, $ \delta_1\delta_2 > 0$ and $k_1, k_2$ are both even (in that case the image misses negative scalar matrices). It is surjective for $m\geq 3$.  
\item On $M_n(\mathbb H)$ the map $\omega$ is surjective for all $m\geq 2$.
\end{enumerate}
\end{theorem}
\noindent Since, in general, the power maps need not be surjective, $m=2$ is the smallest possible with the property that $\omega$ could be surjective. This clearly generalises the known results for the Matrix Waring problem in~\cite{KishoreSingh2022} to a more general diagonal map on one hand, and on the other hand, it generalises the work of Richman~\cite{Richman1985, Richman1987}. The result on $\mathbb H$ for $\delta_i\in \mathbb R$ follows quite easily by using the canonical form theory which we mention in Section~\ref{section-quaternion}. We further hope to apply our method to some more polynomial maps and obtain similar results. Throughout the article $I\in M_n(\mathbb K)$ will denote the $n\times n$ identity matrix in $M_n(\mathbb K)$.



\section{Preliminaries}
In this section, we set up a method that reduces the problem of finding a solution to a matrix equation $f(X_1, \ldots X_m)=A$ for all $A \in M_n(\mathbb K)$ to that of Jordan Matrices over extensions of $\mathbb K$. Recall that an irreducible separable polynomial is one which does not have repeated roots over an algebraic closure. 
 We will call a polynomial to be separable if all of its irreducible factors are separable.
\subsection{Reduction to Jordan Matrices}\label{reduction-jordan}

Let $\mathbb K$ be a field and $A\in M_n(\mathbb K)$. We are interested in finding a solution to the equation $f(X_1, X_2, \ldots, X_m) =A$. That is, we want to know if $A$ is in the image of the map $\omega \colon \mathcal A^m \rightarrow \mathcal A$ given by $(x_1, x_2, \ldots,x_m) \mapsto f(x_1, x_2, \ldots, x_m)$. For $P\in GL_n(\mathbb K)$ we can re-write the equation as 
$$f(PX_1P^{-1}, PX_2P^{-1}, \ldots, PX_mP^{-1}) = PAP^{-1},$$ 
thus it is enough to consider a canonical form of $A$ up to similarity. Further, if we have a block diagonal matrix $A=\diag(A_1, \ldots, A_r)$, and have a solution to each $f(X_1, X_2, \ldots, X_m) = A_i$ for $1\leq i \leq r$, then we have a solution to $f(X_1, X_2, \ldots, X_m) = A$ by taking solution as the block diagonal matrices with blocks being the solutions of the earlier equation. 

Further, let $A$ has its characteristic polynomial separable and is given by $\prod_{i=1}^r p_i(x)^{s_i}$ where $p_i$ are irreducible polynomials then a canonical form of $A$ is $\bigoplus\limits_{i=1}^{r}\bigoplus\limits_{j=1}^{t_i} J_{p_i, s_i(j)}$, where $\{s_i(1),\ldots, s_i(t_i)\}$ is a partition of $s_i$. 
This notation will be abbreviated in the subsequent discussion by omitting the limits when they are clear from the context.
The notation $J_{p, l}$ refers to a block matrix with $l$ blocks each of size $d=deg(p(x))$ and the entries as follows:
$$
J_{p, l} = \begin{psmallmatrix} \mathfrak C_p & I & &&&\\ &\mathfrak C_p& I &&& \\  &&\ddots&\ddots&&\\ &&&\ddots&\ddots&\\&&&&\mathfrak C_p&I\\&&&&&\mathfrak C_p\\  \end{psmallmatrix}   \ \  {\text where} \ \ \mathfrak C_p = \begin{psmallmatrix} 0 & 0& 0& \cdots &0& -a_0\\ 1& 0& 0& \cdots &0& -a_1 \\ 0&1&0&0&\cdots&-a_2\\ \vdots&&\ddots&\ddots&\ddots&\vdots \\ 0&\cdots&0&1&0&-a_{d-2}\\ 0&0& 0& \cdots &1&-a_{d-1}\end{psmallmatrix}
$$
is the companion matrix corresponding to $p(x)=x^d+a_{d-1}x^{d-1}+\cdots +a_1x+a_0$. The matrix $J_{p,l}$ is also referred to as a Jordan matrix corresponding to $p(x)$ and $l$. Let $L=\mathbb K(\alpha)$ be a field extension that has a root $\alpha$ of $p(x)$. Then, we consider the map $\mathcal R \colon M_l(L) \rightarrow M_{ld}(\mathbb K)$ induced by the left-multiplication map by $\alpha$ on $L$ over $\mathbb K$. Clearly, this map is a ring homomorphism and maps the matrix $J_{\alpha, l}$ to $J_{p,l}$. Under the homomorphism $\mathcal R$ we can carry forward a solution of $f(X_1, X_2, \ldots, X_m) = J_{\alpha, l}$ to that of $f(X_1, X_2, \ldots, X_m) = J_{p,l}$. This argument can be used to reduce the problem of finding the solution to that of Jordan matrices $J_{\alpha, l}$ for all $\alpha$ in various extensions of $\mathbb K$. We can summarise this as follows:
\begin{proposition}\label{key-prop}
Let $\mathbb K$ be a field and $A\in M_n(\mathbb K)$ with a separable characteristic polynomial. Then, the equation $f(X_1, X_2, \ldots, X_m) = A$ in $M_n(\mathbb K)$ has a solution if $f(X_1, X_2, \ldots, X_m) = J_{\alpha, l}$ has a solution in $M_l(\mathbb K(\alpha))$ for all $\alpha$ which are eigenvalues of $A$ over $\bar{\mathbb K}$. 
\end{proposition}
\noindent We use this to reduce the problem to solve Equation~\ref{multi-form} to that for Jordan matrices. We demonstrate this by an example.
\begin{example}
Consider $\begin{pmatrix}0&-1&1&0\\ 1&0&0&1\\ &&0&-1 \\&&1&0\end{pmatrix}$ over $\mathbb R$ with characteristic polynomial $(T^2+1)^2$ which becomes $\begin{pmatrix} \iota &1\\ &\iota \end{pmatrix}$ over $\mathbb C$. We can write this as a sum of squares over $\mathbb C$ as follows: 
$$\begin{pmatrix} \iota &1\\ 0&\iota \end{pmatrix} = \begin{pmatrix} \zeta_8 & \zeta_8^{-1}\\ 0&0 \end{pmatrix}^2 + \begin{pmatrix} 0 & 0\\ 0&\zeta_8 \end{pmatrix}^2$$ 
where $\zeta_8 = e^{\frac{2\pi}{8}\iota } = \cos{\frac{2\pi}{8}}+\iota \sin{\frac{2\pi}{8}} = c + \iota s$ which gives us
$$\begin{pmatrix}0&-1&1&0\\ 1&0&0&1\\ &&0&-1 \\&&1&0\end{pmatrix} = \begin{pmatrix}c&-s&c&s\\ s&c&-s&c\\ &&0&0 \\&&0&0\end{pmatrix}^2 +\begin{pmatrix}0&0&0&0\\ 0&0&0&0\\ &&c&-s \\&&s&c\end{pmatrix}^2.$$
\end{example}

Our goal is to show Theorem~\ref{theoremB} which solves Equation~\ref{multi-form}. This requires to understand if the equation $\delta X^{k_1} + \beta Y^{k_2} = A$ has a solution in $M_n(\mathbb K)$ for a given $A\in M_n(\mathbb K)$. We note that without loss of generality, we may assume $\delta =1$. Thus, because of Proposition~\ref{key-prop} our problem gets reduced to the following: 
\begin{proposition}
The equation $X^{k_1} + \beta Y^{k_2} =A$ has a solution in $M_n(\mathbb K)$ if the equation $X^{k_1} + \beta Y^{k_2} =J_{\alpha, l}$ has a solution in $M_l(\mathbb K(\alpha))$ for all $\alpha$ which are eigenvalues of $A$ over $\bar{\mathbb K}$ and $l\geq 1$.
\end{proposition}
\noindent In view of this, we will explicitly deal with the equation
\begin{equation}\label{mainequation} 
X^{k_1} + \beta Y^{k_2} = A \end{equation}
with $\beta \neq 0$ in the following two cases for $A$ and $n \geq 1$:
\begin{itemize}
\item the invertible case when $A=J_{\alpha, n}$ where $\alpha\neq 0$, and 
\item the nilpotent case when $A= J_{0,n}$
\end{itemize}
in the Section~\ref{diagonal-word}. We further note that $A$ is invertible if and only if none of the $p_i(x)$ (appearing in the factorization of the minimal polynomial) are the polynomial $x$, i.e., in the extension fields we need to deal with $J_{\alpha, n}$ with $\alpha\neq 0$. 


\section{Diagonal word}\label{diagonal-word}
In this section, we look at Equation~\ref{mainequation}. Without loss of generality, we may assume $k_1\geq k_2 \geq 2$. Note that if one of the $k_i=1$ then the equation always has a solution.

\subsection{Invertible Elements in the image}
We begin with considering $A=J_{\alpha, n}$, i.e, if $X^{k_1} + \beta Y^{k_2} = J_{\alpha,n}$ with $\alpha\neq 0$ has a solution in $M_n(\mathbb K)$. We have the following:

\begin{lemma}\label{lemma-nonzero-alpha}
Let $k_1, k_2$ be integers and $\alpha\in \mathbb K^*$. Suppose the equation $X^{k_1} + \beta Y^{k_2} = \alpha$ has two solutions $(a,b)$ and $(c,d)$ satisfying $a^{k_1}\neq c^{k_1}$ and  $b^{k_2}\neq d^{k_2}$ over $\mathbb K$. Then, the equation $X^{k_1} + \beta Y^{k_2} = J_{\alpha,n}$ has a solution in $M_n(\mathbb K)_s$ where $M_n(\mathbb K)_s$ denotes the set of diagonalisable matrices.
\end{lemma}
\begin{proof} For $n=1$ we are already given the required solution. So, we may assume $n\geq 2$. 
We have solutions $(a,b), (c,d) \in \mathbb K \times \mathbb K$ such that $a^{k_1} \neq c^{k_1}$ and $b^{k_2}\neq d^{k_2}$, $\alpha=a^{k_1} +\beta b^{k_2}$ and $\alpha=c^{k_1}+\beta d^{k_2}$. With this in mind, we consider the following block diagonal matrices,
\begin{enumerate}
\item when $n$ is even $\mathcal G_n = \underset{\frac{n}{2}}\bigoplus\begin{pmatrix}   a^{k_1} & 1\\    0 & c^{k_1} \end{pmatrix}$ and $\mathcal H_n = {(\beta b^{k_2})} \underset{{\frac{n-2}{2}}}\bigoplus\begin{pmatrix}   \beta d^{k_2} & 1\\
      0 & \beta b^{k_2} \end{pmatrix} \bigoplus {(\beta d^{k_2})}$,
\item when $n$ is odd $\mathcal G_n = \underset{\frac{n-1}{2}}\bigoplus\begin{pmatrix}  a^{k_1} & 1\\   0 & c^{k_1} \end{pmatrix}\bigoplus (a^{k_1})$ and $\mathcal H_n = {(\beta b^{k_2})} \underset{{\frac{(n-1)}{2}}}\bigoplus\begin{pmatrix}  \beta d^{k_2} & 1\\    0 & \beta b^{k_2} \end{pmatrix}$.
\end{enumerate}
Thus, we get $J_{\alpha,n}= \mathcal G_n + \mathcal H_n$. Since $a^{k_1} \neq c^{k_1}$, and $b^{k_2} \neq d^{k_2}$, we get that $\mathcal G_n$ and $\mathcal H_n$ both are diagonalisable matrices, in fact, similar to $\diag\{a^{k_1}, c^{k_1}, a^{k_1}, c^{k_1}, \ldots \}$ and $\beta \diag \{b^{k_2}, d^{k_2},b^{k_2}, d^{k_2}\ldots \}$ respectively. 
Clearly, $\mathcal G_n$ is similar to a matrix which is $k_1$-power of a diagonal matrix, and $\mathcal H_n$ is similar to $\beta$ times $k_2$-power of a diagonal matrix. Hence $J_{\alpha,n}= B^{k_1}+\beta C^{k_2}$ where $B, C \in M_n(\mathbb K)_s$.
\end{proof} 
\noindent We can use Proposition~\ref{key-prop} with this Lemma to show when invertible elements are in the image of the diagonal word map. Suppose $A\in M_n(\mathbb K)$ with a separable characteristic polynomial and each eigenvalue of $A$ over $\bar{\mathbb K}$ satisfies the properties of the Lemma above then $A$ is in the image. 

\begin{remark}\label{remark-nilpotent}
We note that the above proof works for $\alpha =0$ too as long as we have required solutions over $\mathbb K$.  
\end{remark} 
\noindent Here is an example that the image of $\omega$ could have nilpotent elements.
\begin{example}
Consider $\omega \colon M_2(\mathbb K)_s \times M_2(\mathbb K)_s \rightarrow M_2(\mathbb K)$ given 
by $\omega(x_1,x_2) = x_1^2+x_2^2$. Suppose $char(\mathbb K)\neq 2$, and $X^2+1=0$ has a 
solution in $\mathbb K$, say $\iota$, then we can write $\begin{pmatrix} 0& 1\\0&0 \end{pmatrix} = 
\begin{pmatrix} 1 & \frac{1}{2}\\0& 1\end{pmatrix}^2 + \begin{pmatrix} \iota &0\\0&\iota 
\end{pmatrix}^2$. In the case $char(\mathbb K)=2$ we can write $\begin{pmatrix} 0& 1\\0&0 
\end{pmatrix} = \begin{pmatrix} 1 & 0\\0& 0\end{pmatrix}^2 + \begin{pmatrix} 1 
&1\\0& 0 \end{pmatrix}^2$.
\end{example}
\subsection{Jordan nilpotent elements in the image when $n > 2k_1$}

Now we are interested in getting nilpotent elements in the image of $X^{k_1} + \beta Y^{k_2}$. We assume $k_1\geq k_2$. We show that the large-size nilpotents are always in the image. We recall the notion of the Junction matrix from Section 5 of~\cite{KishoreSingh2022}.
\begin{definition}
Let $n\geq 1$ be a positive integer. Let $(n_1, n_2,\ldots , n_k)$ be partition of $n$ with $1 \leq n_1 \leq n_2 \leq \ldots \leq n_k$. The Junction matrix associated with the given partition of n is 
$$\mathfrak{J}_{(n_1,n_2,\ldots ,n_k)}:= e_{n_1,n_1+1}+e_{(n_1+n_2),(n_1+n_2+1)}+ \cdots +e_{(n_1+n_2+\ldots +n_{k-1}),(n_1+n_2+\ldots +n_{k-1}+1)}$$
where $e_{i,j}$ is the matrix with $1$ at $ij^{th}$ place and $0$ elsewhere. 
\end{definition}
\noindent We begin with the following:
\begin{lemma}\label{lemma-junction}
Suppose $n\geq 2k$, and $(n_1, n_2,\ldots , n_k)$ be partition of $n$ with all $n_i\geq 2$. Then, the junction matrix $\mathfrak{J}_{(n_1,n_2,\ldots ,n_k)} = \beta.B^k$ for some $B\in M_n(\mathbb K)$.
\end{lemma}
\begin{proof}
Let $\{e_1, e_2, \ldots, e_n\}$ be the standard basis of $\mathbb K^n$ and the matrix $\mathfrak{J}_{(n_1,n_2,...,n_k)}$ corresponds to a linear transformation given by mapping $e_{(n_1+n_2+\cdots +n_{i+1})}$ to $e_{(n_1+n_2+\cdots + n_i)}$ and others to $0$. Reordering the basis elements to 
\begin{eqnarray*}
&\{e_1,e_2,\ldots ,e_{n_{i-1}}, e_{n_{i+2}}, \ldots, e_n, e_{n_1}, e_{(n_1+1)}, e_{(n_1+n_2)}, e_{(n_1+n_2+1)}, \ldots, \\ & e_{(n_1+n_2+\cdots +n_{k-1})},  e_{(n_1+n_2+\cdots + n_{k-1}+1)}\}
\end{eqnarray*}
gives a conjugate of the junction matrix $\mathfrak{J}_{(n_1,n_2,...,n_k)}$, say $C$. The matrix 
$$C = \left( \underset{n-2(k-1)}\bigoplus J_{0,1}\right) \bigoplus \left(\underset{k-1} \bigoplus J_{0,2}\right).$$
We can also see that $C$ is conjugate to $\beta C$ as $C$ is a nilpotent matrix. Hence, $\mathfrak{J}_{(n_1, n_2, \ldots, n_k)}$ is conjugate to $\beta C$. Now consider $B= \left( \underset{n-(2k-1)}\bigoplus J_{0,1}\right) \bigoplus J_{0,2k-1}$ and by Lemma 6.1 of \cite{KishoreSingh2022}, $B^k$ is conjugate to $C$. Therefore, $\mathfrak{J}_{(n_1,n_2,\ldots, n_k)}$ is conjugate to $\beta B^k$.
\end{proof}

\begin{theorem}\label{theorem-large-nilpotent}
Let $k_1\geq k_2\geq 2$ be positive integers. For $n\geq 2k_1$ the Jordan nilpotent matrix $J_{0,n}$ is in the image of $f(X,Y)= X^{k_1} + \beta Y^{k_2}$. 
\end{theorem}
\begin{proof}
We begin with considering $J_{0,n}^{k_1}$. Let us denote $n'= \lfloor\frac{n}{k_1}\rfloor$ and $n''= \lceil\frac{n}{k_1}\rceil$. 
Since $n\geq 2k_1$, we have $n'' \geq n' \geq 2$. We find $m$ such that $m\equiv n \imod {k_1}$ and $0\leq m\leq k_1$. Then, from Miller's Lemma (Lemma 2~\cite{Miller2016}) we get that $J_{0,n}^{k_1}$ is conjugate to 
$$JF(J_{0,n}^{k_1})=\left(\bigoplus_{k_1-m}J_{0, n'}\right)\bigoplus \left(\bigoplus_{m}J_{0,n''}\right).$$
Now, we consider $\mathfrak J = J_{0,n} - JF(J_{0,n}^{k_1})$. The matrix $\mathfrak J$ is a junction matrix associated to the following partition of $n$: $\left(\underbrace{n', \ldots, n'}_{k_1-m}, \underbrace{n'', \ldots, n''}_{m} \right)$. From the Lemma~\ref{lemma-junction} it follows that $\mathfrak J$ is conjugate to $\beta B^{k_2}$. This completes the proof.
\end{proof}
\noindent Thus, all nilpotent matrices of index $\geq 2k_1$ are in the image.

\subsection{Nilpotent elements in the image}

Now we develop another method to get nilpotent matrices in the image depending on the existence of solutions of certain equations over the base field. 
Let us begin with the following $n$-by-$n$ matrix for $n\geq 3$ and $\epsilon\neq 0$ in $\mathbb K$, 
$$M=M(\epsilon, {\bf x}, {\bf y}, z) = \begin{pmatrix} 
\epsilon J_{0, (n-1)} & ~^t{\bf x} \\ {\bf y} &z \end{pmatrix}$$
over the field $\mathbb K$ where ${\bf x}=(x_1, x_2, \ldots, x_{n-1})$ and ${\bf y}=(y_1, y_2, \ldots, y_{n-1})$ are elements of $\mathbb{K}^{n-1}$ and $~^t{\bf x}$ denotes the transpose of the vector $\bf{x}$. Note that $\epsilon J_{0,(n-1)}$ is conjugate to $J_{0,(n-1)}$. 
The characteristic polynomial of $M$ is 
\begin{eqnarray}\label{charpolyM}
\chi_M(T) &=& T^n-zT^{n-1} - \left(\sum_{i=1}^{n-1} x_iy_i\right)T^{n-2} - \epsilon\left(\sum_{i=2}^{n-1} x_iy_{i-1}\right)T^{n-3} - \cdots \\ 
&-&  \epsilon^{j-2}\left(\sum_{i=j-1}^{n-1} x_iy_{i-j+2}\right)T^{n-j} - \cdots - \epsilon^{n-3}(x_{n-2}y_1 + x_{n-1}y_2)T  - \epsilon^{n-2}x_{n-1}y_1. \nonumber
\end{eqnarray}
We wish to understand when $M$ is a $k$ power regular semisimple element (i.e., with distinct diagonal entries). We recall that the elementary symmetric polynomials are 
$$\mathcal E_i(X_1,\ldots, X_n)=\sum_{1\leq j_1 < j_2 < \cdots < j_i \leq n} X_{j_1} X_{j_2}\cdots X_{j_i}$$
and $\prod_{r=1}^n(T-x_r)=T^n-\mathcal E_1(x_1,\dots,x_n)T^{n-1}+\cdots +(-1)^n \mathcal E_n(x_1,\dots,x_n)$. 
\begin{definition}
Let $\mathbb K$ be a field and $\lambda_1, \ldots, \lambda_{n}$ be a solution of $X_1^k+X_2^k+\cdots + X_{n}^k = \alpha$ over $\mathbb K$. We say the solution is {\bf regular} if  $\lambda_i^k\neq \lambda_j^k$ for $i \neq j$. Further, if none of the $\lambda_i$ are $0$ we say the solution is non-zero regular. 
\end{definition}
\noindent Note that if $0$ appears in a regular solution then it can appear at most once. We have the following,
\begin{lemma}\label{lemma-nilpotent}
Let $n\geq 3$, and $\mu_1, \ldots, \mu_{n}$ be a regular solution of $X_1^k + X_2^k + \cdots + X_{n}^k = z$ over $\mathbb K$. Then, for a given ${\bf y}$ with $y_1 \neq 0$ (similarly for a given ${\bf x}$ with $x_{n-1} \neq 0$) there exists ${\bf x}$ (respectively ${\bf y}$) such that the matrix $M(\epsilon, {\bf x}, {\bf y}, z)$ is conjugate to the $k$ power regular semisimple element $\diag(\mu_1^k, \ldots, \mu_n^k)$.
\end{lemma}
\begin{proof} Let us write $M=M(\epsilon, {\bf x}, {\bf y}, z)$ and we have expression for $ \chi_M(T)$ in the Equation~\ref{charpolyM}. 
We require $\chi_M(T)=(T-\mu_1^k)(T-\mu_2^k)\cdots (T-\mu_n^k)$ which leads to having a solution to the following system of equations: 
\begin{eqnarray*}
z &=& \mathcal E_{1}(\mu_1^k, \ldots, \mu_n^k)= \sum \mu_i^k\\
x_1y_1+\cdots +x_{n-1}y_{n-1} &=& - \mathcal E_{2}(\mu_1^k, \ldots, \mu_n^k)\\
x_{2}y_1 + x_{3}y_2+\cdots +x_{n-1}y_{n-2} &=& \epsilon^{-1}\mathcal E_{3}(\mu_1^k, \ldots, \mu_n^k)\\
\vdots &=&\vdots \\
x_{n-2}y_1 +x_{n-1}y_2 &=& (-1)^{n} \epsilon^{-(n-3)}\mathcal E_{n-1}(\mu_1^k, \ldots, \mu_n^k)\\
 x_{n-1}y_1 &=& (-1)^{n+1} \epsilon^{-(n-2)} \mathcal E_n(\mu_1^k, \ldots, \mu_n^k)=(-1)^{n+1}\epsilon^{-(n-2)}\prod\mu_i^k.
\end{eqnarray*}
Clearly, the first equation is satisfied. We write the remaining equations in matrix form as follows: 
$$\begin{pmatrix} y_1& y_2& \cdots & y_{n-1} \\ 0& y_1 & \cdots & y_{n-2} \\ \vdots &&\ddots & \vdots\\ 0& 0& \cdots & y_1\end{pmatrix} \begin{pmatrix}x_1\\ x_2\\ \vdots\\ x_{n-1} \end{pmatrix}=\begin{pmatrix} - \mathcal E_{2}(\mu_1^k, \ldots, \mu_n^k)\\ \\ \vdots \\ (-1)^{n+1}\epsilon^{-(n-2)}\mu_1^k\cdots \mu_n^k\end{pmatrix}.$$
Since $y_1\neq 0$ the matrix is invertible and we have a solution. 

Now, we can also write the above equations taking the bottom one first and thinking of $y_i$ as variables as follows:
$$\begin{pmatrix} x_{n-1}& 0& \cdots & 0\\ x_{n-2}& x_{n-1} & \cdots & 0 \\ \vdots &&\ddots & \vdots\\ x_1& x_2& \cdots & x_{n-1}\end{pmatrix} \begin{pmatrix}y_1\\ y_2\\ \vdots\\ y_{n-1} \end{pmatrix}=\begin{pmatrix} (-1)^{n+1}\epsilon^{-(n-2)}\mu_1^k\cdots \mu_n^k \\ \\ \vdots \\ - \mathcal E_{2}(\mu_1^k, \ldots, \mu_n^k)\end{pmatrix}.$$
Since $x_{n-1}\neq 0$ we have a solution to this equation.
\end{proof}

\begin{corollary}\label{cor1-nilpotent}
Let $n\geq 3$, $\lambda_1, \ldots, \lambda_{n}$ be a regular solution of $X_1^k+X_2^k+\cdots + X_{n}^k = 1$ over $\mathbb K$ and $\epsilon\in \mathbb K^*$ with $\epsilon\neq -1$. Then, the matrix $M(\epsilon, (1+\epsilon){\bf e}_{n-1}, {\bf y}, 1)$ is a $k$ power regular semisimple element where $y_i= \frac{(-1)^{n-i}}{(1+\epsilon)}\mathcal E_{n-i+1}(\lambda_1^k, \ldots, \lambda_{n-1}^k)$ and ${\bf e}_{n-1}=(0, \ldots, 0,1)$.
\end{corollary}

\begin{corollary}\label{cor2-nilpotent}
Let $n\geq 3$, $\mu_1, \ldots, \mu_{n}$ be a regular solution of $X_1^k+X_2^k+\cdots + X_{n}^k = -\frac{1}{\beta}$ over $\mathbb K$ with $\mu_n=0$. Then, for a given ${\bf y}$ with $y_1\neq 0$ there exists ${\bf x}$ with $x_{n-1}=0$ such that the matrix $M(1, {\bf x}, {\bf y}, -1)$ is conjugate to $\beta$ times a $k$ power regular semisimple element.
\end{corollary}
\begin{proof} 
The proof follows along the similar lines as for the Lemma~\ref{lemma-nilpotent} by equating $\chi_M(T)$ to $(T-\beta \mu_1^k)\cdots (T-\beta \mu_n^k)$ by noting that the equation $x_{n-1}y_1=0$ will ensure $x_{n-1}=0$.
\end{proof}

\begin{theorem}\label{nilpotent-n3}
Let $n\geq 3$ and $\mathbb K$ be a field with $|\mathbb K|>2$ and suppose 
\begin{enumerate}
\item the equations $X_1^{k_1} +X_2^{k_1}+\cdots + X_{n}^{k_1} = 1$ has a regular solution, and  
\item in addition, for $n\geq 3$, $X_1^{k_2} +X_2^{k_2}+\cdots + X_{n-1}^{k_2} = -\frac{1}{\beta}$ has a non-zero regular solution. 
\end{enumerate}
Then, the nilpotent matrix $J_{0,n}$ is in the image of $f(X,Y)= X^{k_1} + \beta Y^{k_2}$. 
\end{theorem}
\begin{proof}
It is enough to show that the Jordan nilpotent matrix $J_{0,n}$ is in the image of $f(X,Y)= X^{k_1} + \beta Y^{k_2}$. Let $\epsilon \in \mathbb K^*$ such that $1+\epsilon \neq 0$.
We write 
$$M(\epsilon, (1+\epsilon){\bf e}_{n-1}, {\bf y}, 1) + M(1, {\bf x}, -{\bf y}, -1) = \begin{pmatrix} (1+\epsilon) J_{0, (n-1)} & (1+\epsilon){~^t \bf e}_{n-1}+{~^t\bf x} \\ 0& 0\end{pmatrix}$$
where ${\bf e}_{n-1}=(0, \ldots, 0,1)$ and ${\bf x}=(x_1, \ldots, x_{n-2}, 0)$. The matrix on the right-hand side is conjugate to $J_{0,n}$. The first matrix on the left side is $k_1$th power of a diagonalisable matrix (follows from Corollary~\ref{cor1-nilpotent}). It also ensures $y_1\neq 0$. The second matrix is $\beta$ times a $k_2$th power of a diagonalisable matrix (follows from Corollary~\ref{cor2-nilpotent}).
\end{proof}

\begin{theorem}\label{nilpotent-n2}
Let $\mathbb K$ be a field with $|\mathbb K|>2$. Suppose the equation $X_1^{k_2} +X_2^{k_2} = -\frac{1}{\beta}$ has a regular solution. Then, the nilpotent matrix $J_{0,2}$ is in the image of $f(X,Y)= X^{k_1} + \beta Y^{k_2}$. 
\end{theorem}
\begin{proof}
Let $(\lambda_1,\lambda_2)$ be a regular solution of $X_1^{k_1}+X_2^{k_1}=-\frac{1}{\beta}$. We write 
$$\begin{pmatrix}0&1\\0&0\end{pmatrix} = \begin{pmatrix}0&0\\y&1\end{pmatrix} + \beta \begin{pmatrix}0&\frac{1}{\beta}\\-\frac{y}{\beta}&-\frac{1}{\beta}\end{pmatrix}.$$
Now, the first matrix on the right is an idempotent and hence is $k_1$-th power of itself. For the second matrix to conjugate to $\diag(\lambda_1^{k_2}, \lambda_2^{k_2})$, we need $\lambda_1^{k_2}+\lambda_2^{k_2} = -\frac{1}{\beta}$ and $\lambda_1^{k_2}\lambda_2^{k_2}=\frac{y}{\beta}$. Determining $y$ gives us the desired result.
\end{proof}
\section{Image of diagonal polynomial over $\mathbb C$ and $\mathbb F_q$}\label{diagonal-fieldC}
We use the results obtained in the previous section to obtain some surjectivity results over particular fields for the diagonal word map. 

\subsection{Diagonal polynomial over $\mathbb C$}
We demonstrate an application of our earlier results. Note that this result will still hold for any algebraically closed field of characteristic $0$, but we choose to state it for $\mathbb C$.
\begin{theorem}\label{diagonal-form-C}
Let $\mathbb K=\mathbb C$ and $k_1, k_2 \geq 1$ be integers and $\beta$ be a non-zero element in $\mathbb C$. Then, the map $\omega \colon M_n(\mathbb C)_s \times  M_n(\mathbb C)_s \rightarrow M_n(\mathbb C)$ given by $\omega(x_1,x_2)= x_1^{k_1} + \beta x_2^{k_2}$, where $M_n(\mathbb C)_s$ is the set of semisimple matrices, is surjective.  
\end{theorem}
\begin{proof}
Let $A$ be in $M_n(\mathbb C)$. Then, the Jordan canonical form of $A$ is the direct sum of the Jordan matrices $J_{\alpha, l}$ where $\alpha \in \mathbb C$. Now, we look at the equation $X^{k_1} + \beta Y^{k_2} = \alpha$ over $\mathbb C$. Take any $a,c \in \mathbb C$ such that $a^{k_1}\neq c^{k_1}$ and consider the equations $\beta Y^{k_2}=\alpha-a^{k_1}$ and $\beta Y^{k_2}=\alpha-c^{k_1}$. We can easily find solutions required in the Lemma~\ref{lemma-nonzero-alpha}, thus $J_{\alpha, l}$ is in the image. The same argument works for Jordan nilpotent matrices in the view of remark~\ref{remark-nilpotent} (or alternatively we can use Theorem~\ref{nilpotent-n3}, \ref{nilpotent-n2} for $\alpha=0$ case). This proves the required result.
\end{proof}
\begin{corollary}
Let $m\geq 2$, and $\omega \colon M_n(\mathbb C)^m \rightarrow M_n(\mathbb C)$ given by $\omega(x_1, x_2, \ldots,x_m)=\delta_1 x_1^{k_1} + \delta_2 x_2^{k_2} + \cdots +\delta_m x_m^{k_m}$. Then, $\omega(M_n(\mathbb C)_s)=M_n(\mathbb C)$. Thus, $\omega$ is surjective. 
\end{corollary}

\subsection{Over finite field}

The result in this section is a generalisation of that in~\cite{KishoreSingh2022}. The proof is along similar lines too thus we keep it short.

The proof relies on having enough solutions of the equation $X^{k_1}+\beta Y^{k_2} = c$ over the field $\mathbb F_q$, for large enough $q$. The solution of polynomial equations over finite fields has a long history with some fundamental results such as the Chevalley-Warning theorem and Lang-Weil bound etc. We begin with some of these results regarding the number of solutions which will be used in the main proof. We recall a version of Lang-Weil theorem \cite[Theorem 5A]{Schmidt2004}.
\begin{theorem}\label{Lang-Weil}
Consider the polynomial equation $\delta_1X_1^{k_1} + \cdots + \delta_mX_m^{k_m} =1$ where $\delta_i\in \mathbb F_q^*$ and $k_i>0$ for all $i$. Then the number of solutions $S$ of this equation in $\mathbb F_q^m$ satisfies 
$$|S-q^{m-1}| \leq k_1k_2\cdots k_m q^{\frac{m-1}{2}} \left(1-\frac{1}{q}\right)^{-m/2}.$$
\end{theorem}
\noindent The next lemma is along a similar line as Proposition A.3~\cite{KishoreSingh2022}.
\begin{lemma}\label{lemma1}
For $k_1 \geq k_2\geq 2$ and $\alpha, \beta \in \mathbb F_q^\times$, consider the polynomial $$F(X_1,X_2)=X_1^{k_1}+ \beta X_2^{k_2} -\alpha.$$ 
Then, for $q > k_1^4 k_2^4$, there exists solutions $(a,b)$ and $(c,d)$ to $F(X_1, X_2)=0$ such that $a^{k_1}\neq c^{k_1}$ and $b^{k_2}\neq d^{k_2}$.
\end{lemma}
\begin{proof}
By Theorem~\ref{Lang-Weil}, we have the following inequality about the number of solutions $S$ of the equation $F(X_1, X_2)=0$ and $m=2$,
 $$|S-q| \leq k_1k_2\sqrt{q}\left(\frac{q}{q-1}\right).$$
Observe that $\frac{q}{q-1} \leq 2\leq k_1$. Therefore, we have $|S - q| \leq k_1^2k_2\sqrt{q}$. Suppose $(a, b)$ and $(c, d)$ are solutions of $F(X_1, X_2)=0$. If $a^{k_1}=c^{k_1}$, then $F(X_1, X_2) = 0$ has at most $k_2^2$ solutions as $(a, \zeta_{k_2}b)$ and $(c, \zeta_{k_2}d)$, where $\zeta_{k_2}$ refers to a root of unity if it exists, are also the possibility for solutions. Similarly, for $b^{k_2}=d^{k_2}$, there are at most $k_1^2$ solutions possible. So, we need to have 
$$S \geq q- k_1^2k_2 \sqrt{q}  \geq  k_1^2+k_2^2+1$$ 
i.e., we want $\sqrt{q}(\sqrt{q} - k_1^2k_2)\geq k_1^2 + k_2^2 + 1$.
For this to be satisfied, it suffices to have $\sqrt{q} > k_1^2k_2^2$. In that case, we get $\sqrt{q}(\sqrt{q} - k_1^2k_2) > \sqrt{q} > k_1^2k_2^2 \geq 4k_1^2 \geq k_1^2+k_2^2+1$ as $k_2\geq 2.$
\end{proof}
\begin{corollary}\label{cor-diag-invert}
Let $k_1, k_2 \geq 1$ be integers and $\alpha\in \mathbb F_q^*$. Then, there exist a constant $\mathcal N_1$ (depending on $k_1$ and $k_2$ only) such that for all $q > \mathcal N_1$, the matrix $J_{\alpha, n}\in M_n(\mathbb F_q)$ can be written as $B^{k_1}+\beta C^{k_2}$ for some $B, C \in M_n(\mathbb F_q)$ both diagonalisable.
\end{corollary}
\begin{proof}
Using Lemma~\ref{lemma1}, there exists a constant $\mathcal N_1$ (depending on $k_1, k_2$ only) such that for $q > \mathcal N_1$, and $\alpha\in \mathbb F_q^{\times}$, there exist solutions $(a,b), (c,d) \in \mathbb F_q^2$ such that $a^{k_1} \neq c^{k_1}$ and $b^{k_2}\neq d^{k_2}$, $\alpha=a^{k_1} +\beta b^{k_2}$ and $\alpha=c^{k_1}+\beta d^{k_2}$ for $q > \mathcal N_1$. Now we can simply use the Lemma~\ref{lemma-nonzero-alpha} to get the required solution. 
\end{proof} 
Now, we recall Proposition 2.3 from~\cite{Kishore2022} and Proposition A.2 from~\cite{KishoreSingh2022} which guarantees regular solutions to certain equations over $\mathbb F_q$.  
\begin{lemma}
Let $\gamma \in \mathbb F_q^*$ and $n \geq 2$ be an integer. Then, there exists a constant $\mathcal N_2$, depending on $k$ and $n$, such that for all $q  > \mathcal N_2$ the equation $X_1^k + X_2^k +\cdots + X_n^k =\gamma$ has a regular solution over $\mathbb F_q$. In fact, it always has a non-zero regular solution when $n\geq 3$.   
\end{lemma}

\begin{proposition}\label{prop-nilpotent-small}
Let $|\mathbb K| > 2$. For every integer $k_1\geq k_2\geq 1$, and $\beta\in \mathbb F_q^*$ there exists a constant $\mathcal N_3$, depending on $k_1, k_2$ and $n$ only, such that for all $q> \mathcal N_3$ the Jordan nilpotent matrix $J_{0,n}$ is in the image of $X^{k_1} + \beta Y^{k_2}$.   
\end{proposition}
\begin{proof}
In view of Lemma above the required hypothesis of Theorem~\ref{nilpotent-n2} and~\ref{nilpotent-n3} are satisfied if $q > \mathcal N_3$. Note that $\mathcal N_3$ is the maximum of the constants required in the hypothesis of the referred Theorems for various choices of $k_1$ and $k_2$ for different $n$. Thus, we have the required result.  
\end{proof}
\noindent Now we are ready to prove the main result of this section,
\begin{theorem}\label{diagonal-form-q}
Let $k_1, k_2 \geq 1$ and $n\geq 2$ be integers and $\beta$ be a non-zero element in the finite field $\mathbb F_q$. Consider the map $\omega \colon M_n(\mathbb F_q) \times  M_n(\mathbb F_q) \rightarrow M_n(\mathbb F_q)$ given by $\omega(x_1,x_2)= x_1^{k_1} + \beta x_2^{k_2}$. Then, there exists a constant $\mathcal N(k_1, k_2)$ (which depends only on $k_1$ and $k_2$) such that for all $q > \mathcal N(k_1, k_2)$, the map $\omega$ is surjective.  
\end{theorem}
\begin{proof}
In the view of Proposition~\ref{key-prop} the problem is reduced to dealing with $J_{\alpha, l}$ for all extensions of $\mathbb F_q$ where $l\leq n$. The case of $\alpha\neq 0$ is covered by Corollary~\ref{cor-diag-invert} for all $q > \mathcal N_1$ where $\mathcal N_1$ depends on $k_1$ and $k_2$ only. The case of $J_{0,l}$ for $l > 2k_1$ is covered by Theorem~\ref{theorem-large-nilpotent} which works for any $q$. For the case of $J_{0,l}$ with $l\leq 2k_1$ we use Proposition~\ref{prop-nilpotent-small} which works for $q > \mathcal N_2$ depending on $k_1, k_2$ and $l$ as well. Thus, if we take $q > \mathcal N$ where $\mathcal N$ is the maximum of $\mathcal N_1$ and various $\mathcal N_2$ for $l < 2k_1$ (which are finitely many) we get the result. Note that $\mathcal N$ depends on $k_1$ and $k_2$ only. 
\end{proof}
\begin{corollary}
For $m\geq 2$, there exists a constant $\mathcal N$ depending only on $k_1, \ldots, 
k_m$ such that for all $q > \mathcal N$ the map $\omega \colon M_n(\mathbb F_q)^m \rightarrow M_n(\mathbb F_q)$ given by $\omega(x_1, x_2, \ldots,x_m) \mapsto \delta_1 x_1^{k_1} + \delta_2 x_2^{k_2} + \cdots +\delta_m x_m^{k_m}$ is surjective.
\end{corollary}

\section{Image of diagonal word over $\mathbb R$}\label{diagonal-fieldreal}

In this section, we consider the diagonal polynomials with coefficients in $\mathbb R$ and look at its image over $M_n(\mathbb R)$. Our main theorems in this section are as follows:
\begin{theorem}\label{diagonal-form-r1}
Let $\mathbb K = \mathbb R$, $k_1 \geq k_2 \geq k_3 \geq 1$ be integers and $\beta, \gamma$ be non-zero elements in $\mathbb R$. Then, the map $\omega \colon M_n(\mathbb R)^3 \rightarrow M_n(\mathbb R)$ given by $\omega(x_1, x_2, x_3)= x_1^{k_1} + \beta x_2^{k_2}+ \gamma x_3^{k_3}$ is surjective.
\end{theorem}
\begin{theorem}\label{diagonal-form-r}
Let $\mathbb K = \mathbb R$, $k_1 \geq k_2 \geq 1$ be integers and $\beta > 0$ in $\mathbb R$. Then, the map $\omega \colon M_n(\mathbb R) \times  M_n(\mathbb R) \rightarrow M_n(\mathbb R)$ given by $\omega(x_1,x_2)= x_1^{k_1} + \beta x_2^{k_2}$ is surjective if and only if one of the following holds 
\begin{enumerate}
\item[(i)] $n$ is even, 
\item[(ii)] $n$ is odd and one of the $k_1$ or $k_2$ is odd.
\end{enumerate}
Further, when $n$ is odd and $k_1, k_2$ both are even the image is $M_n(\mathbb R) \setminus \{\lambda I_n \mid \lambda < 0\}$. 
\end{theorem}
\noindent We may assume that all of the coefficients of the diagonal polynomial are positive. That is, we are dealing with $\delta_1 x_1^{k_1} + \delta_2 x_2^{k_2} + \cdots +\delta_m x_m^{k_m}$ where $\delta_i >0$ real for all $i$. Because for $x_1^{k_1} + \beta x_2^{k_2}$ with $\beta < 0$, following a similar argument as for $\mathbb C$ in Section~\ref{diagonal-form-C}, the equations required in the Theorem~\ref{nilpotent-n3}, \ref{nilpotent-n2} have solutions over $\mathbb R$, hence the map given by $x_1^{k_1} + \beta x_2^{k_2}$ would be surjective. In fact, without loss of generality, we may assume that $\delta_i = 1$ as $\delta_i >0 $ has a $k_i$-th root. Thus in what follows, we will be dealing with the map given by $x_1^{k_1} + x_2^{k_2} + \cdots + x_m^{k_m}$. The rest of the section is devoted to the proof of these statements. 

We begin by recalling a result from Richman (see Theorem 6 \cite{Richman1985}) which also uses the work of Griffin and Krusemeyer from~\cite{GK1977}:
\begin{theorem}[Richman, Griffin-Krusemeyer]
Let $k$ be a field with characteristic not equal to $2$ and $n$ be odd. Then, a scalar matrix $cI_n \in M_n(\mathbb K)$ is a sum of two squares if and only if $c$ is a sum of two squares in $\mathbb K$.
\end{theorem}
\noindent Thus in view of this, we have,
 \begin{corollary}\label{Richman-negative}
Let $n$ be odd and $k_1, k_2$ both even. Suppose $\beta > 0$ is a real number. Then, a scalar matrix $\lambda I_n \in M_n(\mathbb R)$ for $\lambda < 0$ can not be written as $A^{k_1}+ \beta B^{k_2}$ where $A, B\in M_n(\mathbb R)$. 
\end{corollary} 
\begin{proof}
If we can write $\lambda I_n \in M_n(\mathbb R)$ as $A^{k_1}+ \beta B^{k_2}$ then $\lambda I_n \in M_n(\mathbb R)$ is also a sum of two squares in $M_n(\mathbb R)$, and then by the above Theorem of Richman $\lambda$ is a sum of two squares. This is not possible for $\lambda < 0$.
\end{proof}
The rest of the proof is devoted to essentially showing that these are the only exceptions. The proof will be divided into three cases:
\begin{enumerate}
\item[Case 1:] When one of the $k_1$ or $k_2$ is odd.
\item[Case 2:] Both $k_1$ and $k_2$ are even and $n$ is even.
\item[Case 3:] Both $k_1$ and $k_2$ are even and $n$ is odd.
\end{enumerate}

\subsection{Case 1 when one of the $k_1$ or $k_2$ is odd:}\label{subsection-k1k2odd} 

The proof when $k_1$ or $k_2$ is odd is simpler. Let $A\in M_n(\mathbb R)$. Then, $A$ is conjugate to the direct sum of Jordan blocks 
\begin{enumerate}
\item $J_{\alpha, l}$ where $\alpha \geq 0$ in $\mathbb R$, 
\item $J_{\alpha, l}$ where $\alpha < 0 $ in $\mathbb R$, and 
\item $J_{p(x), l}$ where $p(x)$ is degree $2$ irreducible polynomial over $\mathbb R$. 
\end{enumerate}
Using Proposition~\ref{key-prop} we can realise $J_{p(x), l}$ of the kind $J_{\lambda ,l}$ for some $\lambda\in \mathbb C$ where we can use Theorem~\ref{diagonal-form-C} to prove the result. Thus, we need to deal with $J_{\alpha, l}$ where $\alpha\in \mathbb R$. We may assume $k_2$ is odd. Note that when $\alpha \neq 0$, we are done using Lemma~\ref{lemma-nonzero-alpha} as the equation $X^{k_1}+ Y^{k_2}=\alpha$ has required solutions (in the view of $k_2$ being odd). 

When $\alpha=0$ we can use Theorem~\ref{nilpotent-n3} and~\ref{nilpotent-n2} to get the result as we have solutions of required kind over $\mathbb R$ (once again in view of $k_2$ being odd).

\subsection{Case 2 when $n$ is even and $k_1, k_2$ both even:}\label{subsection-n-even}

In this case, we wish to show that the equation  $p(x_1, x_2) = x_1^{k_1} + x_2^{k_2} = A$ where $A$ is in $M_{n}(\mathbb R)$ and $n$ is even, always has a solution in $M_n(\mathbb R)$. Since the equation is closed under conjugation we may consider $A$ in its canonical form. The blocks appearing in the canonical form of $A$ will be as follows:
\begin{enumerate}
\item $J_{\alpha, 2m}$ where $\alpha \in \mathbb R$ and $m\geq 1$.
\item $J_{\alpha_1, 2m_1 -1}\oplus J_{\alpha_2, 2m_2 -1}$ for $\alpha_1, \alpha_2 \in \mathbb R$ and $m_1, m_2 \geq 1$ (because $n$ is even).
\item $J_{f(x), m}$ where $f$ is an irreducible polynomial of degree $2$ over $\mathbb R$, and the particular case $m=1$ refers to the $2\times 2$ companion matrix $\mathfrak C_f$ .
\end{enumerate}

Denote $\tau_m = \underbrace{\left(\tau \oplus \tau \oplus \cdots \oplus \tau \right)}_{m-times}$ where $\tau = \begin{pmatrix} \cos\frac{\pi}{k_1} & -\sin\frac{\pi}{k_1}\\ \sin\frac{\pi}{k_1} & \cos\frac{\pi}{k_1}
\end{pmatrix}$ then $\tau_m^{k_1} = -I_{2m}$. Thus, for a real number $\xi >0$, $\left(\xi^\frac{1}{k_1} \tau_m\right)^{k_1} = -\xi I_{2m}$. Also, we note that for $\xi >0$ and $\eta = \xi^{\frac{1}{k_2}}$, the Jordan matrix $J_{\eta, s}$ has the property that $\left(J_{\eta, s}\right)^{k_2}$ is conjugate to $J_{\xi, s}$. 

In view of the discussion above we need to deal with the three kinds of blocks. In each of these cases, we show that the matrix is in the image of $\omega(x_1, x_2)$.
\begin{enumerate}
\item When $\alpha > 0$, we know $J_{\alpha, 2m}$ has $k_1$-th root, so we are done. When $\alpha \leq 0$ pick $\psi > 0$ such that $\psi + \alpha > 0$, we write
$$J_{\alpha, 2m} = - \psi I_{2m} + J_{(\alpha +\psi), 2m}= \left(\psi^{\frac{1}{k_1}} T_m\right)^{k_1} + J_{(\alpha +\psi), 2m}.$$
Now, note that $J_{\alpha +\psi, 2m}$ is a $k_2$-th power as it is a conjugate of $\left(J_{\sqrt[k_2]{(\alpha +\psi)}, 2m}\right)^{k_2}$. 

\item In the case of $J_{\alpha_1, 2m_1 -1}\oplus J_{\alpha_2, 2m_2 -1}$ if both $\alpha_1, \alpha_2$ are positive it has $k_1$-th root. Else, we pick $\psi >0$ such that $\psi + \alpha_1 >0$ and $\psi +\alpha_2 >0$ and write 
$$J_{\alpha_1, 2m_1 -1} \oplus J_{\alpha_2, 2m_2 -1} = -\psi I_{2(m_1+m_2-1)} + (J_{\alpha_1 +\psi, 2m_1 -1}\oplus J_{\alpha_2 + \psi, 2m_2 -1}).$$
Now, $-\psi I_{2(m_1+m_2-1)}$ is a $k_1$-th root and $J_{\alpha_1 +\psi, 2m_1 -1}\oplus J_{\alpha_2 + \psi, 2m_2 -1}$ is a $k_2$-th root.
\item In this case we can use Lemma~\ref{lemma-nonzero-alpha} and go over the extension $\mathbb C$ where we can find the solution and realise it over $\mathbb R$.
\end{enumerate}

\subsection{Case 3 when $n$ is odd and $k_1, k_2$ both even:}\label{subsection-n-odd}

In this case, we wish to show that the equation  $p(x_1, x_2) = x_1^{k_1} + x_2^{k_2} = A$ where $A$ is in $M_{n}(\mathbb R)$ and $n$ is odd, has a solution in $M_n(\mathbb R)$ except when $A$ is a negative scalar matrix. The fact that negative scalars are not in the image follows from Corollary~\ref{Richman-negative}. Since the equation is closed under conjugation we can work with the canonical form of $A$. The canonical form of $A$ is a direct sum of the following:
\begin{enumerate}
\item $J_{\alpha_1, l_1} \oplus J_{\alpha_2, l_2}\oplus \cdots \oplus J_{\alpha_r, l_r}$ where $\alpha_i \in \mathbb R$ and $l_i$ even.
\item $J_{\alpha_1, l_1} \oplus J_{\alpha_2, l_2}\oplus \cdots \oplus J_{\alpha_s, l_s}$ where $\alpha_i \in \mathbb R$ and $l_i$ odd ($s$ must be odd as $n$ is odd).
\item $\oplus J_{f_i(x), m}$ where $f_i$ is an irreducible polynomial of degree $2$ over $\mathbb R$. The particular case $i=1, m=1$ refers to the $2\times 2$ companion matrix $\mathfrak C_f$ .
\end{enumerate}

We first prove some Lemma. 
\begin{lemma}\label{lemma-n-odd1} 
Let $m\geq 3$. Let $d_1, d_2, \ldots, d_m$ be non-negative reals. Consider the matrix  
$$M=\begin{pmatrix} -d_1 & 1 &&&& \\ &-d_2& 1& && \\ &&-d_3 & 1 &&\\ &&&\ddots&\ddots&\\ &&& &-d_{m-1} & 1 \\ -a_1 &-a_2 & \cdots&\cdots & -a_{m-1}& -(d_m +1)\end{pmatrix} \in M_m(\mathbb R).$$
Then, $a_j$s can be chosen in such a way that the matrix $M$ is $k$-th power of some diagonalizable matrix in $M_m(\mathbb R)$.
\end{lemma}
\begin{proof}
The characteristic polynomial of $M$ is given by $\chi_M(T)$
\begin{align*}
=& T^m + T^{m-1}\left(1+ \sum_{1\leq i \leq m}d_{i} \right) + T^{m-2} \left(\sum_{1\leq i_1 < i_2\leq m}d_{i_1} d_{i_2} + \sum_{1\leq i\leq m-1}d_{i}+a_{m-1} \right) +\\
& T^{m-3}\left(\sum_{1\leq i_1 < i_2 < i_3 \leq m} d_{i_1}d_{i_2}d_{i_3} + \sum_{1\leq i_1< i_2 \leq m-1} d_{i_1}d_{i_2} + a_{m-1}\left(\sum_{1\leq i \leq m-2}d_{i}\right) + a_{m-2}\right) +\\
& \cdots  + T\left(\sum_{1\leq i_1<i_2<i_3< \ldots <i_{m-1}\leq m} d_{i_1} d_{i_2}\ldots d_{i_{m-1}}+ \cdots + a_3\left(d_1 + d_2\right)+ a_2\right) +\\ 
& \left(d_1d_2\ldots d_{m-1}(d_m +1)+ d_1d_2\ldots d_{m-2}a_{m-1}+ \ldots +d_1a_2 + a_1 \right). 
\end{align*}
We claim that we can choose $\lambda_1,\ldots, \lambda_{m-2}$ positive reals and $\lambda_{m-1}, \lambda_m$, a pair of non-real complex conjugates such that 
$\chi_M(T)=(T-\lambda_1)\cdots (T-\lambda_{m-2})(T-\lambda_{m-1})(T-\lambda_m)$ and $\lambda_i \neq \lambda_j$ for all $i\neq j$. This will help us ensure that $M$ is conjugate to a regular semisimple element which is $k$-th power. For this, we need to solve the following system of equations,
\begin{equation*}
\begin{split}
\sum_{1\leq i \leq m}d_{i} + 1 & = - \mathcal E_1(\lambda_1, \lambda_2\ldots, \lambda_m),\\
\sum_{1\leq i_1 < i_2\leq m} d_{i_1}d_{i_2} + \sum_{1\leq I \leq m-1}d_{i} + a_{m-1} & = (-1)^2 \mathcal E_2(\lambda_1, \ldots, \lambda_{m-1},\lambda_m),
\end{split}
\end{equation*}
\begin{equation*}
\begin{split} 
\sum_{1\leq i_1 < i_2 < i_3 \leq m} d_{i_1}d_{i_2}d_{i_3} + \sum_{1\leq i_1< i_2\leq m-1}d_{i_1}d_{i_2} & + a_{m-1}\left(\sum_{1\leq i \leq m-2} d_{i}\right) + a_{m-2}\\  &= (-1)^3\mathcal E_3(\lambda_1, \ldots, \lambda_{m-1},\lambda_m),
\end{split}
\end{equation*}
\begin{equation*}
\begin{split}
\vdots\\
\sum_{1\leq i_1<i_2< i_3< \ldots <i_{m-1}\leq m} d_{i_1}d_{i_2}\ldots d_{i_{m-1}} + \ldots &+ a_3\left(d_1+d_2\right)+a_2 \\ & = (-1)^{m-1}\mathcal E_{m-1}(\lambda_1, \ldots,  \lambda_{m-1},\lambda_m),\\
d_1d_2\ldots d_m+ d_1d_2\ldots d_{m-1}+ \ldots +a_2d_1+a_1 &= (-1)^m \mathcal E(\lambda_1, \ldots, \lambda_{m-1},\lambda_m). \\
\end{split}
\end{equation*} 

Now, we simply pick $\lambda_1,\ldots, \lambda_{m-2}$ positive reals and $\lambda_{m-1}, \lambda_m$ a pair of non-real complex conjugate such that
$$-\sum_{1\leq i \leq m}d_{i} - 1  =  \mathcal E_1(\lambda_1, \ldots, \lambda_{m-2}, \lambda_{m-1},\lambda_m) = \sum_{i=1}^{m-2} \lambda_i + \lambda_{m-1}+\lambda_m $$ and $\lambda_i\neq \lambda_j$ for all $i\neq j$. This ensures each $\lambda_i$ for $1\leq i \leq m-2$ has a real $k$-th root. Further note that when $k$ is even the real part of $\lambda_m$ has to be negative and we will see that the requirement of $k$-th root forces a pair of complex conjugates. 

We note that since $\lambda_{m-1}$ and $\lambda_m$ are a pair of complex conjugate the $\mathcal E_j(\lambda_1, \ldots, \lambda_m)$ are in $\mathbb R$. Thus, the above equations determine the value of $a_j$ which are real numbers. This proves that we can choose $a_j$ in such a way that $M$ is conjugate to $\tilde M = diag(\lambda_1, \ldots, \lambda_{m-2})\oplus \mathfrak C_f$ where $\lambda_i$ are positive reals and $\mathfrak C_f$ is companion matrix of the real polynomial $T^2-(\lambda_{m-1}+\lambda_m)T+\lambda_{m-1}\lambda_m$. Clearly, $\tilde M$ has a $k$-th root in $M_m(\mathbb R)$ because scalars on the diagonal are positive and the $2\times 2$ block can be thought of as an element in $\mathbb C$. Hence $M$ is a $k$-th power.
\end{proof}

\noindent Next we have the following,
\begin{lemma}\label{lemma-n-odd2}
Let $n\geq 3$ and $d_1, \ldots, d_n$ be non-negative reals. Consider the matrix $T\in M_n(\mathbb R)$ as follows:
$$T=\begin{pmatrix} - d_1 &1&&&& \\ & - d_2&1& && \\ &&\ddots&\ddots&&\\ &&&\ddots&\ddots&\\ &&&& - d_{n-1}&1\\ &&&&& - d_n \end{pmatrix}.$$
Then, there exist matrices $B$ and $C$ in $M_n(\mathbb R)$ such that $T= B^{k_1} + C^{k_2}$.
\end{lemma}
\begin{proof} 
Let us consider a matrix $B = \begin{pmatrix} 0&0&\cdots&&0\\ 0&0&\cdots&&0 \\ a_1 & a_2&\cdots & a_{n-1} &1\end{pmatrix}$ and observe that $B^{k_1} = B$. Now, from Lemma~\ref{lemma-n-odd1} we can choose $a_1, \ldots, a_{n-1}$ such that the matrix $T-B$, which is in the required form, is $k_2$-th power, say $C^{k_2}$. Thus $T = B^{k_1} + C^{k_2}$.
\end{proof}
\begin{proposition}\label{prop-step1}
Let $n$ be odd and $A\in M_n(\mathbb R)$. Suppose the odd-size Jordan blocks appearing in the canonical form of $A$ are all of the size $\geq 3$. Then, $A$ is in the image of $x_1^{k_1} + x_2^{k_2}$. 
\end{proposition}
\begin{proof}
The Jordan blocks of the kind $J_{\alpha, 2l}$ and $J_{f(x), l}$ can be taken care of as in the Subsection~\ref{subsection-n-even} as they are of even size. For the Jordan blocks of the kind $J_{\alpha, 2l-1}$ we can use Lemma~\ref{lemma-n-odd2} as they are of size $\geq 3$. 
\end{proof}
\begin{lemma}\label{lemma-diagonal}
All diagonal matrices in $M_m(\mathbb R)$, when $m$ is even, and the diagonal matrices with at least $2$ distinct diagonals, when $m$ is odd, are in the image of  $x_1^{k_1} + x_2^{k_2}$.
\end{lemma}
\begin{proof}
When $m$ is even the diagonal matrices belonging to the image are covered in Section~\ref{subsection-n-even}. Now for $m$ odd, we are done if even one of the diagonal entries is positive. So, we may assume all of the diagonal entries are negative. We need to deal with $\diag(\lambda_1, \lambda_2, \lambda_3, \ldots)$ where $\lambda_i < 0$ for all $i$ and $\lambda_1 \neq \lambda_2$. Note that $\diag(\lambda_1, \lambda_2, \lambda_3, \ldots)$ is similar to $\begin{pmatrix} \lambda_1 & 1& & \\  & \lambda_2 & & \\ &  & \lambda_3 & \\ &&&\ddots\end{pmatrix}$ (because $\lambda_1 \neq \lambda_2$). 

Now, we consider the matrix $L = \begin{pmatrix}  0 & 0 \\ a_1 & a_0 \end{pmatrix}$ and note that 
$$ M:= \begin{pmatrix} \lambda_1 & 1\\  & \lambda_2 \end{pmatrix} - L^{k_1} = \begin{pmatrix} \lambda_1 & 1 \\ -a_1a_0^{k_1-1} & \lambda_2 - a_0^{k_1} \end{pmatrix}.$$
We claim that we can choose $a_0$ and $a_1$ such that the characteristic polynomial of the above matrix $M$ is $\chi_M(T) = (T+\lambda_1)(T-\mu)$ with $\mu \neq -\lambda_1$. For this we need to have $-\lambda_1 + \mu = \lambda_1 + \lambda_2-a_0^{k_1}$ and $-\lambda_1 \mu = \lambda_1(\lambda_2-a_0^{k_1})+ a_1a_0^{k_1-1}$. The first equation would require $a_0^{k_1} = 2\lambda_1+\lambda_2 -\mu$ which can be solved by choosing $\mu$ so that $2\lambda_1+\lambda_2 -\mu > 0$ (this means $\mu < 2\lambda_1+\lambda_2 < 0 \leq -\lambda_1$). The second equation gives $a_1$. Thus, we can make $M$ similar to $\diag(-\lambda_1, \mu)$ with $\mu \neq -\lambda_1$. Hence, 
$$\tilde M:= \begin{pmatrix} \lambda_1 & 1& & \\  & \lambda_2 & & \\ &  & \lambda_3 & \\ &&&\ddots\end{pmatrix} - \begin{pmatrix} L&&\\ &0& \\ &&\ddots \end{pmatrix}^{k_1} $$ is similar to $\diag(-\lambda_1, \mu, \lambda_3, \ldots)$ where $-\lambda_1 > 0$. Now, $\tilde M$ has one of the diagonal entries positive (and remaining part is even size) so it is a $k_2$-th power by the earlier argument. This completes the proof.  
\end{proof}

\begin{lemma}\label{lemma-even-oddsize}
Let $l>1$ and $\alpha, \xi \in \mathbb R$. The matrices $J_{\alpha, 2l}\oplus (\xi)$ and $J_{f, l}\oplus (\xi)$ are in the image of $x_1^{k_1} + x_2^{k_2}$.
\end{lemma}
\begin{proof}
First we deal with $J_{\alpha, 2l}\oplus (\xi)$. From the argument in Section~\ref{subsection-n-even} we know that $J_{\alpha, 2l}$ is in the image of $x_1^{k_1} + x_2^{k_2}$ and hence if $\xi > 0$ (it has roots) we are done. Thus, we may assume $\xi <0$. 

Write $m=2l$ for simplicity and consider $w_1, w_2, \ldots, w_m \in \mathbb R$ such that $w_1\neq 0$. Take $L$ to be a matrix with first $m-1$ rows $0$ and the last row to be $(w_m, w_{m-1}, \ldots, w_1)$. Then,  
$$J_{\alpha, m} - L^{k_1}= \begin{pmatrix} \alpha & 1 & 0 & \cdots & 0\\    0 & \alpha & 1 & \cdots & 0\\    \vdots & \vdots & \ddots&  \ddots & \vdots\\  \vdots &\cdots&\cdots&\alpha&1\\ -w_{m} w_1^{k_1-1} & -w_{m-1} w_1^{k_1-1} & \cdots & \cdots & \alpha-w_1^{k_1} \end{pmatrix}.
$$
We claim that we can choose $w_1, w_2, \ldots, w_m$ in such a  way that the characteristic polynomial of $J_{\alpha, m} - L^{k_1}$ is $(X + \xi)^{m-1}(X-\lambda)$ with $\lambda \neq -\xi$. Following a similar calculation as in the proof of Lemma~\ref{lemma-n-odd1}, we need to ensure $tr(J_{\alpha, m} - L^{k_1}) = -(m-1) \xi + \lambda$, that is, $m\alpha - w_1^{k_1} = -(m-1) \xi + \lambda$. Thus, we need to have a solution for $w_1^{k_1} = m\alpha +(m-1) \xi - \lambda$ which can be insured by choosing $\lambda < 0$, in fact take $\lambda < \xi <0$. This allows us that $J_{\alpha, m} - L^{k_1}$ is conjugate to $M_1 \oplus (\lambda)$ where only eigen values of $M_1$ is $-\xi$ which is positive. Thus, $M_1 = M^{k_2}$ for some $M$. Now, let us write $u^{k_1}= \xi - \lambda > 0$ and $\lambda = -v^{k_1}$. Then, $\begin{pmatrix} J_{\alpha, m} & \\ &\xi \end{pmatrix} - \begin{pmatrix} L & \\ & u \end{pmatrix}^{k_1} = \begin{pmatrix} J_{\alpha, m} -L^{k_1} & \\ &\xi -u^{k_1} \end{pmatrix} $ is conjugate to $\begin{pmatrix} M^{k_2} & & \\ &\lambda & \\ & &\lambda \end{pmatrix}$. Now all we need to show is that $\lambda I_2$ is a $k_2$-th power where $\lambda <0$. Equivalently, enough to show $-I_2$ is a $k_2$-th power. This can be done using $-I=\begin{pmatrix}
     \cos\frac{\pi}{k_2}& -\sin\frac{\pi}{k_2}\\
     \sin\frac{\pi}{k_2} & \cos\frac{\pi}{k_2}
 \end{pmatrix}^{k_2}$.

Now, let us deal with $J_{f,l}\oplus (\xi)$. Once again, from the argument in Section~\ref{subsection-n-even} we know that $J_{f, l}$ is in the image of $x_1^{k_1} + x_2^{k_2}$ and hence if $\xi > 0$ we are done. We need to deal with the case when $\xi <0$. 

First, we consider the case of $l=1$ with $f(x)=x^2 + b_0x +b_1$ and $\mathfrak C_f:=\begin{pmatrix} 0 & -b_1\\ 1& -b_0 \end{pmatrix} \in M_2(\mathbb R)$. Let $w_1,w_2$ be two real numbers with $w_1$ non-zero. Consider the matrix $L=\begin{pmatrix} 0&0\\ w_2 & w_1 \end{pmatrix}$. Then, 
$$ \begin{pmatrix} \mathfrak C_f & \\ & \xi \end{pmatrix} - \begin{pmatrix} L & \\ & u \end{pmatrix}^{k_1} = \begin{pmatrix} 0 & -b_1 & 0\\ 1-w_2w_1^{k_1-1} & -b_0-w_1^{k_1} & 0\\ 0 & 0 & \xi -u^{k_1} \end{pmatrix}$$ 
and we claim that with a choice of $w_1, w_2$ and $u$ we can make this a $k_2$-th power. Note that the characteristic polynomial of $\mathfrak C_f -L^{k_1}$ is $T^2 + (b_0 +w_1^{k_1})T + b_1(1-w_2w_1^{k_1-1})$. To make $\mathfrak C_f -L^{k_1}$ conjugate to a diagonal matrix $\diag(-\xi, \lambda)$ with $\lambda \neq -\xi$ we need to equate trace and determinant, i.e, $-\xi + \lambda = -b_0 - w_1^{k_1}$ and $-\xi\lambda = b_1(1-w_2w_1^{k_1-1})$. We fix, $\lambda < 0$ such that $\xi - \lambda > 0$ and $-b_0 + \xi -\lambda >0$, which ensures, solution for $u^{k_1} = \xi - \lambda > 0$ (to get $u$) and $w_1^{k_1} = -b_0 + u^{k_1}$. The second equation gives $w_2$ and with this choice of $u, w_1$ and $w_2$ we get:
$$ \begin{pmatrix}\mathfrak C_f & \\ & \xi \end{pmatrix} - \begin{pmatrix} L & \\ & u \end{pmatrix}^{k_1} = \begin{pmatrix} -\xi & & \\ & \lambda & \\  &  & \lambda \end{pmatrix}=:M_2.$$ 
Now, $-\xi >0$ which has $k_2$-th root and $\lambda <0$ we can use the earlier trick on $2\times 2$ block by using $k_2$-th root of $-I$ to get the job done.

Now for $l>1$, let $L$ be a size $2$ matrix as above for some $w_1$ and $w_2$. Consider 
$$ Q:=\begin{pmatrix} \mathfrak C_f & I  &  &&  & \\
 & \mathfrak C_f & I &  &  & \\  & & \ddots & \ddots & & \\
 &  &  & \mathfrak C_f & I & \\  &  &  & & \mathfrak C_f & \\
  &  &  & &  & \xi \end{pmatrix} - \begin{pmatrix}
0 & 0  & 0 &\hdots& 0 & 0\\  0 & 0 & 0 & \hdots & 0 & 0\\
0 & 0 & 0& \hdots & 0 & 0\\ \vdots & \vdots & \vdots & \vdots & \vdots & \vdots\\
0 & 0 & 0 &\hdots & L & 0\\ 0 & 0 & 0 & \hdots & 0 & u
\end{pmatrix}^{k_1}.$$
Then, the characteristic polynomial of $Q$ is $f(T)^{l-1}(T+ \xi)(T-\lambda)^2$ equal to the minimal polynomial and hence $Q$ is similar to $J_{f, l-1} \oplus M_2$. Since $J_{f,m-1}$ is $k_2-th$ power and $M_2$ is $k_2$-th power so is $Q$. This completes the proof.
\end{proof}

 \subsection{Proof of the Theorem}
 \begin{proof}[{\bf Proof of Theorem~\ref{diagonal-form-r}}]
When $n$ is even the proof follows from the argument in Section~\ref{subsection-k1k2odd} and~\ref{subsection-n-even}. 
Now, suppose $n \geq 3$ is odd and $k_1, k_2$ is both even. The negative real scalar matrices are not in the image, as follows from Corollary~\ref{Richman-negative}. It remains to show that every matrix in $M_n(\mathbb R)$ is in the image of $x_1^{k_1}+ x_2^{k_2}$ unless it is of the form $-\lambda I_{n}$ where $\lambda$ is a positive real. 

Let $A\in M_n(\mathbb R)$. From Proposition~\ref{prop-step1} if all odd-size Jordan blocks appearing in the canonical form of $A$ are of size $\geq 3$ then we are done. So, we may assume that the odd-size Jordan blocks are of size $1$. Since $n$ is odd there has to be at least one block of size $1$. Now we claim that if $A$ has at least $1$ even size Jordan block then $A$ is in the image. For this we can combine one even-size block with a size $1$ block and use Lemma~\ref{lemma-even-oddsize} and the remaining parts will be even-size blocks. Thus, we are left with the case when $A$ has no blocks of even size (and no blocks of odd size $\geq 3$). That is, $A$ is a diagonal matrix. Once again if $A$ has at least $2$ distinct entries on the diagonal we are done with Lemma~\ref{lemma-diagonal}. Thus, $A$ must be a scalar matrix of the form $\lambda I_n$. 

 \end{proof}

\begin{proof}[{\bf Proof of Theorem~\ref{diagonal-form-r1}}]
From the proof above all we need to show that $\lambda I_n$ when $\lambda < 0$ and $n$ odd is of the form $x_1^{k_1} + x_2^{k_2} + x_3^{k_3}$. For this we write $\lambda I_n = \diag(\lambda, \ldots, \lambda, 0) + \diag(0, \ldots, 0, \lambda)$. Here, the first one is $k_1$-th power using $\tau_{\frac{n-1}{2}}$ from Section~\ref{subsection-n-even}. For the second one again we use the argument from Section~\ref{subsection-n-even} on $\begin{pmatrix} 0&\\ & \lambda\end{pmatrix}$ and write it as a sum of $k_2$ and $k_3$ powers.
\end{proof}


\section{Images of diagonal polynomial over real quaternions}\label{section-quaternion}
In this section, we look at the diagonal polynomial over $M_n(\mathbb H)$ where $\mathbb H$ is Hamilton's real quaternion division algebra. We show that the map $\omega(x_1,x_2)=x_1^{k_1}+\beta x_2^{k_2}$ is surjective when $\beta\neq0$. This easily implies the surjectivity of the diagonal map for all $m\geq 2$. The result here is surprisingly easy to obtain due to the canonical form theory for matrices in $M_n(\mathbb H)$. We begin with the following result due to Wiegmann and Liping (See \cite[Theorem 1]{Wieg55}, also \cite[Lemma 3]{Li01}).
\begin{lemma}\label{lem:Conjugacy-over-quarternion}
Every $n \times n$ matrix with real quaternion elements is similar
under a matrix transformation with real quaternion elements to a matrix in (complex) Jordan normal form with diagonal elements of the form $a + bi$, $b \geq 0$. That is to say if $A\in M_n(\mathbb H)$, then $A$ is similar to a matrix of the form
\begin{align*}
J(A):=J_{\lambda_1,n_1}\oplus J_{\lambda_2,n_2}\oplus\ldots\oplus J_{\lambda_k,n_k},
\end{align*}
with $\lambda_k = a_k + ib_k\in\mathbb{C}$ being right eigenvalues of $A$. Furthermore, $b_k\in\mathbb{R} $ can be chosen to be non-negative. 
In this decomposition $J(A)$ is uniquely determined by $A$ up to the
order of Jordan blocks $J_{\lambda_k,n_k}$, and $J(A)$ is said to be the Jordan canonical form of $A$ corresponding to maximal subfield $\mathbb C$ of $\mathbb H$.
\end{lemma}
\noindent This Lemma reduces the problem to look for $A\in M_n(\mathbb H)$ as an image of diagonal polynomial to that of $A\in M_n(\mathbb C)$.
{\color{black} We call a matrix $A\in M_n(\mathbb H)$ to be invertible if there exists $B\in M_n(\mathbb H)$ such that $AB=BA=I$.
Note that any matrix $A\in M_n(\mathbb H)$ has finitely many conjugacy classes of left eigenvalues (i.e. $\alpha\in \mathbb H$ such that $A\cdot v=  \alpha\cdot v$ for some $v\in {\mathbb H}^n$). 
Since the number of conjugacy classes in $\mathbb H$ are infinite, given any matrix $A\in M_n(\mathbb H)$, there exists $\lambda\in \mathbb H$ such that $\lambda$ is not a left eigenvalue of $A$, and consequently 
$A-\lambda\cdot I$ is invertible (see \cite[Proposition 5.3.4]{Rodman14}).
Now we are ready to state and prove the final result of this article.
\begin{theorem}
   Let $0\neq \beta\in\mathbb H$. Then the map $\omega(x_1,x_2)=x_1^{k_1}+ \beta x_2^{k_2}$ is surjective on $M_n(\mathbb H)$. In particular, any matrix in $M_n(\mathbb H)$ can be written as a sum of two $k$-th powers.
\end{theorem}
\begin{proof}
If $A\in M_n(\mathbb H)$ is invertible it can be written as $X^{k_1}$ for some matrix $X$ in $M_n(\mathbb H)$, since it is so in $M_n(\mathbb C)$.
Hence $A$ is in the image of $\omega$, by setting $x_2$ to be the zero matrix.
Next, assume $A\in M_n(\mathbb H)$ is not invertible.
Choose $\lambda \in \mathbb H$ such that $\lambda^{k_2}$ does not belong to the set of left eigenvalues of $A$. 
Fix $\eta\in \mathbb H$ such that $\eta^{k_2}=\beta$. 
This can be done as $\eta$ can be conjugated to a complex number, thanks to Lemma~\ref{lem:Conjugacy-over-quarternion}.
Then for $x_2=(\lambda/\eta)\cdot I$, the matrix $A-\beta x^{k_2}$ is an invertible matrix (since $0$ is not a left eigenvalue of $A-\beta x^{k_2}$) and hence the map is surjective by the previous argument.
\end{proof}}
\color{black}
\printbibliography
\end{document}